\documentclass[12pt]{amsproc}
\usepackage{amssymb}
\usepackage{amsmath}
\usepackage{amsfonts}

\newtheorem{theorem}{Theorem}

\newtheorem{corollary}[theorem]{Corollary}

\newtheorem{definition}[theorem]{Definition}
\newtheorem{example}[theorem]{Example}

\newtheorem{lemma}[theorem]{Lemma}

\newtheorem{proposition}[theorem]{Proposition}
\newtheorem{remark}[theorem]{Remark}

\input{tcilatex}

\begin{document}
\title[Dickson algebras are atomic at $p$]{Dickson algebras are atomic at $p$%
}
\author{Nondas E. Kechagias}
\address{Department of Mathematics, University of Ioannina,\\
Greece, 45110}
\email{nkechag@uoi.gr}
\urladdr{http://www.math.uoi.gr/\symbol{126}nondas\_k }
\thanks{This paper is in final form and no version of it will be submitted
for publication elsewhere.}

\begin{abstract}
The notion of atomicity defined by Cohen, Moore and Neisendorfer is studied
for the Dickson algebras. Not any ring of invariants respects this property.
It depends on the property of the Dickson algebra that given any monomial $d$
there exists a sequence of Steenrod operations $P\left( \Gamma ,d\right) $
such that $P\left( \Gamma ,d\right) d$ becomes a $p$-th power of the top
Dickson algebra generator.\ \ 
\end{abstract}

\maketitle

\section{Statement of results}

The term atomic was introduced by Cohen, Moore and Neisendorfer in (\cite%
{C-M-N}) to answer the question of whether a given space admits a nontrivial
product decomposition up to homotopy. We consider the analogue question for
the Dickson algebra. Let $V^{n}$ denote an $n$-dimensional vector space over 
$%
%TCIMACRO{\TeXButton{TeX field}{\mathbb Z}}%
%BeginExpansion
\mathbb Z%
%EndExpansion
/p%
%TCIMACRO{\TeXButton{TeX field}{\mathbb Z}}%
%BeginExpansion
\mathbb Z%
%EndExpansion
$, then 
\begin{equation*}
H^{\ast }\left( BV^{n};%
%TCIMACRO{\TeXButton{TeX field}{\mathbb Z}}%
%BeginExpansion
\mathbb Z%
%EndExpansion
/p%
%TCIMACRO{\TeXButton{TeX field}{\mathbb Z}}%
%BeginExpansion
\mathbb Z%
%EndExpansion
\right) \cong E(x_{1},\cdots ,x_{n})\otimes P[y_{1},\cdots ,y_{n}].
\end{equation*}%
$P[y_{1},\cdots ,y_{n}]^{GL_{n}}=%
%TCIMACRO{\TeXButton{TeX field}{\mathbb Z}}%
%BeginExpansion
\mathbb Z%
%EndExpansion
/p%
%TCIMACRO{\TeXButton{TeX field}{\mathbb Z}}%
%BeginExpansion
\mathbb Z%
%EndExpansion
\left[ d_{n,1},\cdots ,d_{n,n-1},d_{n,n}\right] $ denotes the classical
Dickson algebra which is a graded polynomial algebra and 
\begin{equation*}
D_{n}:=\left( E(x_{1},\cdots ,x_{n})\otimes P[y_{1},\cdots ,y_{n}]\right)
^{GL_{n}}
\end{equation*}
the extended Dickson algebra studied by Mui (\cite{Mui}).

In a series of papers, Campbell, Cohen, Peterson and Selick studied self
maps on certain important spaces in topology. In order to prove homotopy
equivalence at $p$, they considered the corresponding $mod-p$ homology
homomorphisms. So they had to use the Dyer-Lashof algebra and (or) quotients
of it as the main ingredient.\ It is well known that the Dyer-Lashof algebra
is closely related with Dickson algebras. Hence, they studied and used
properties of the Dickson algebras, especially papers \cite{C-P-S1} and \cite%
{C-C-P-S1}. The advantage is that the Steenrod algebra action on Dickson
algebras is better understood than the Nishida relations on the Dyer-Lashof
algebra. \ 

Motivated by topological questions regarding the cohomology of an infinite
loop space and strongly influenced by the work of Campbell, Cohen, Peterson
and Selick in \cite{C-P-S1} and \cite{C-C-P-S1} we study the problem under
which conditions is a degree preserving $\mathcal{A}$-endomorphism of $D_{n}$
an isomorphism. Here $\mathcal{A}$ stands for the Steenrod algebra.\ 

\textbf{Theorem.} \ref{Th4} \textit{The extended Dickson algebra }$D_{n}$%
\textit{\ is atomic as a Steenrod algebra module.}\ 

The proof depends on a remarkable property that $D_{n}$ satisfies with
respect to its Steenrod algebra action.

\textbf{Theorem.} \ref{power of d-n,n}\ \textit{Let }$d^{K}=\prod%
\limits_{1}^{n}d_{n,i}^{k\left( i\right) }$\textit{. There exists a sequence
of Steenrod operations }$P\left( \Gamma ,K\right) $\textit{\ such that }%
\begin{equation*}
P\left( \Gamma ,K\right) d^{K}=ud_{n,n}^{p^{m}}\text{.}
\end{equation*}%
\textit{for some natural number }$m$\textit{\ and unit }$u$\textit{.}

In particular,

\textbf{Corollary.} \ref{decomposit}\ $\overline{D_{n}}$\textit{\ is not
directly decomposable as an \ }$A$\textit{-module.\ }

Here $\overline{D_{n}}$\ denotes the augmentation ideal of $D_{n}$ and
corollary implies that the only direct summands are $0$ and $\overline{D_{n}}
$.

Finally we apply Theorem \ref{Th4} to the study of self maps between $%
Q_{0}S^{0}$.

\textbf{Theorem.} \ref{isomorphism in loops} \textit{Let }$%
f:Q_{0}S^{0}\rightarrow Q_{0}S^{0}$\textit{\ be an }$H$\textit{-map which
induces an isomorphism on }$H_{2p-3}(Q_{0}S^{0};%
%TCIMACRO{\TeXButton{Z-p}{\Bbb{Z}/\Bbb{Z}p}}%
%BeginExpansion
\Bbb{Z}/\Bbb{Z}p%
%EndExpansion
)$\textit{. Let }$p>2$\textit{\ and }%
\begin{equation*}
f_{\ast }(Q^{\left( p,1\right) }[1])=uQ^{\left( p,1\right) }[1]+others
\end{equation*}%
\textit{for some }$u\in (%
%TCIMACRO{\TeXButton{Z-p}{\Bbb{Z}/\Bbb{Z}p}}%
%BeginExpansion
\Bbb{Z}/\Bbb{Z}p%
%EndExpansion
)^{\ast }$\textit{. Then }$f_{\ast }$\textit{\ is an isomorphism.}

Last Theorem has been proved by Campbell, Cohen, Peterson and Selick in \cite%
{C-C-P-S1} for $p=2$.\textit{\ }

We recently informed that Pengelley and  Williams have studied similar
properties of the Dickson algebra in \cite{Pengelley}. 

\section{Classical Dickson algebras are atomic at $p$}

The term atomic was defined by Cohen, Moore and Neisendorfer. Let us recall
from \cite{C-P-S1} that a CW complex $X$ is atomic at $p$, if given any map $%
f:X\rightarrow X$\ such that $f$\ induces an isomorphism on $H_{r}(X,%
%TCIMACRO{\TeXButton{Z}{\Bbb{Z}}}%
%BeginExpansion
\Bbb{Z}%
%EndExpansion
/p%
%TCIMACRO{\TeXButton{Z}{\Bbb{Z}}}%
%BeginExpansion
\Bbb{Z}%
%EndExpansion
)$, then $f_{(p)}$\ is a homotopy equivalence. Here $r$ is the lowest degree
such that $\overline{H}_{r}(X,%
%TCIMACRO{\TeXButton{Z}{\Bbb{Z}}}%
%BeginExpansion
\Bbb{Z}%
%EndExpansion
/p%
%TCIMACRO{\TeXButton{Z}{\Bbb{Z}}}%
%BeginExpansion
\Bbb{Z}%
%EndExpansion
)\neq 0$.

In this section we study the action of the Steenrod algebra on monomials of
the Dickson algebra. We prove that its augmentation ideal is directly
indecomposable. First, we recall definitions and basic properties for the
benefit of the reader. The main result of this section is Theorem \ref{power
of d-n,n}.\ 

The following theorem is well known:

\begin{theorem}
\cite{Dic} The classical Dickson algebra is 
\begin{equation*}
P[y_{1},\cdots ,y_{n}]^{GL_{n}}=%
%TCIMACRO{\TeXButton{TeX field}{\mathbb Z}}%
%BeginExpansion
\mathbb Z%
%EndExpansion
/p%
%TCIMACRO{\TeXButton{TeX field}{\mathbb Z}}%
%BeginExpansion
\mathbb Z%
%EndExpansion
\left[ d_{n,1},\cdots ,d_{n,n-1},d_{n,n}\right] \text{.}
\end{equation*}%
It is a polynomial algebra and $\left\vert d_{n,i}\right\vert =2\left(
p^{n}-p^{n-i}\right) $ ($\left\vert d_{n,i}\right\vert =2^{n}-2^{n-i}$ for $%
p=2$).
\end{theorem}

As $n$ varies, we get Dickson algebras of various length. In the last
section, Dickson algebras of different lengths will be considered in
connection with the $mod-p$ cohomology of the base point of $%
QS^{0}=\lim \Omega ^{n}\Sigma ^{n}S^{0}$. \ 

We shall recall some well known results concerning the action of the
Steenrod algebra on Dickson algebra generators.

\begin{proposition}
\label{actionp-th}\cite{Kech1} (Th. 30, p. 169) 
\begin{equation*}
P^{p^{k}}(d_{n,i}^{p^{j}})=\left\{ 
\begin{array}{lll}
d_{n,i+1}^{p^{j}}\text{,} & \text{if }k=n+j-i-1\text{ and }i<n & \text{(1)}
\\ 
-d_{n,i}^{p^{j}}d_{n,1}^{p^{j}}\text{,} & \text{if }k=j+n-1 & \text{(2)} \\ 
0\text{, } & \text{otherwise} & 
\end{array}%
\right. \text{.}
\end{equation*}
\end{proposition}

\begin{definition}
Let $c$ and $j$\ be natural numbers such that $j\leq c+1$. Let $P(c,j)$
stand for the Steenrod iterated operation\ \ \ \ 
\begin{equation*}
P\left( c,j\right) =P^{p^{c-j+1}}...P^{p^{c}-j+j}\text{.}
\end{equation*}
\end{definition}

\begin{lemma}
Let $i=1,...,n$\ and $k\left( i\right) $\ a natural number. Let $c=k\left(
i\right) +n-i-1$ for $i<n$\ and $c=k\left( n\right) +n-1$. Then 
\begin{equation*}
P\left( c,j\right) (d_{n,i}^{p^{k\left( i\right) }})=\left\{ 
\begin{array}{ll}
ud_{n,n}^{p^{k\left( i\right) }}\text{,} & \text{if }0<n-i=j\text{ and }i<n
\\ 
ud_{n,n}^{2p^{k\left( i\right) }}\text{,} & \text{if }i=j=n \\ 
d_{n,i+j}^{p^{k\left( i\right) }}\text{, } & i+j<n \\ 
0, & \text{otherwise}%
\end{array}%
\right. \text{.}
\end{equation*}%
Here $u$\ is a unit.
\end{lemma}

\begin{proof}
This is a repeated application of last proposition. If $k\left( i\right)
+n-1=k\left( n\right) +n-1$, then $k\left( i\right) =k\left( n\right) $ and $%
k\left( i\right) +n-i-1<k\left( n\right) +n-1$ for $1\leq i<n$. Hence case
(2) of last proposition applies only for $i=n$.
\end{proof}

\begin{lemma}
\label{SteenActionLemma3}Let $i=1,...,n$\ and $k\left( i\right) $\ natural
numbers such that $c=k\left( i\right) +n-i-1$ \ and $c=k\left( n\right) +n-1$%
. Let $m\left( i\right) =a_{i}p^{k\left( i\right) }$ with $0\leq a_{i}\leq
p-1$\ and $j=\max \left\{ n,n-i\ |\ a_{n},a_{i}\neq 0\right\} $. Let $%
d=\prod\limits_{1}^{n}d_{n,i}^{m\left( i\right) }$. Then 
\begin{equation*}
P\left( c,j\right) (d)=\left\{ 
\begin{array}{ll}
ud_{n,n}^{p^{k\left( n\right) }}d\text{,} & \text{if }j=n \\ 
ud_{n,n}^{p^{k\left( n-j\right) }}dd_{n,n-j}^{-p^{k\left( n-j\right) }}\text{%
,} & \text{if }j<n%
\end{array}%
\right. \text{.}
\end{equation*}%
Here $u$\ is a unit. Moreover, $P\left( c,t\right) (d)=0$,\ if $j<t$.\ 
\end{lemma}

\begin{proof}
The proof depends on the Cartan formula and last lemma. Two cases shall be
considered: $j=n$\ and $j<n$.

First case: $j=n$. 
\begin{equation*}
P^{p^{c}}d=u_{n}d_{n,1}^{p^{k\left( n\right)
}}d+\sum\limits_{1}^{n-1}u_{i}d_{n,i+1}^{p^{k\left( i\right)
}}dd_{n,i}^{-p^{k\left( i\right) }}\text{.}
\end{equation*}%
Here $u_{i}=0$, if $a_{i}=0$. Since $k\left( n\right) <k\left( 1\right)
<...<k\left( n-1\right) $, we have 
\begin{equation*}
P^{p^{c-1}}\left( d_{n,1}^{p^{k\left( n\right) }}d\right) =u_{n}^{\prime
}d_{n,2}^{p^{k\left( n\right) }}d\text{.}
\end{equation*}%
And 
\begin{equation*}
P\left( c-1,n-1\right) \left( d_{n,1}^{p^{k\left( n\right) }}d\right)
=u^{\prime }d_{n,n}^{p^{k\left( n\right) }}d\text{.}
\end{equation*}%
For the rest of the terms we have 
\begin{equation*}
P^{p^{c-1}}\left( d_{n,i+1}^{p^{k\left( i\right) }}dd_{n,i}^{-p^{k\left(
i\right) }}\right) =u_{i}^{\prime }d_{n,i+2}^{p^{k\left( i\right)
}}dd_{n,i}^{-p^{k\left( i\right) }}\text{.}
\end{equation*}%
Since $n>i$, last lemma implies 
\begin{equation*}
P\left( c-1,n-1\right) \left( d_{n,i+1}^{p^{k\left( i\right)
}}dd_{n,i}^{-p^{k\left( i\right) }}\right) =0\text{.}
\end{equation*}%
The proof of the case $j<n$ follows a similar pattern.

For the last statement of the lemma, proposition \ref{actionp-th} is
applied: For $j<t\leq n$ and $c-j=k\left( n-j\right) -1<k\left( n-j\right)
+j-1=k\left( i\right) +n-i-1$, we have\ \ 
\begin{equation*}
P^{p^{c-j}}\left( d_{n,n}^{p^{k\left( n-j\right) }}dd_{n,n-j}^{-p^{k\left(
n-j\right) }}\right) =0\text{.}
\end{equation*}
\end{proof}

We are ready to proceed to the main Theorem of this section which is the
building block to construct an algorithm turning a monomial $d$\ to $%
d_{n,n}^{p^{l}}$. Let us firstly demonstrate our method. \ 

\begin{example}
Let $p=2\ $and $n=3$. Instead of $d_{3,1}$, $d_{3,2}$\ and $d_{3,3}$\ we
write $d_{1}$, $d_{2}$\ and $d_{3}$\ respectively. Let $%
K=(k_{1}=2^{2}+2^{3},k_{2}=2^{3}+2^{4},k_{3}=2^{1}+2^{2})$ and 
\begin{equation*}
d^{K}=d_{1}^{2^{2}+2^{3}}d_{2}^{2^{3}+2^{4}}d_{3}^{2+2^{2}}\text{.}
\end{equation*}%
Let $J$ be the sequence consisting of the first terms of those of $K$. 
\begin{equation*}
J=(2^{3},2^{2},2^{1})\text{.}
\end{equation*}%
For each exponent $m_{i}$ in $J$ consider the minimum of $m_{n}+n-1$ and $%
m_{i}+n-i-1$ for $1\leq i\leq n-1$: 
\begin{equation*}
\min \left( J\right) =\left\{ 3+0,2+1,1+2\right\} =3\text{.}
\end{equation*}%
Let $i_{\left( J\right) }$ be the maximum of the following set 
\begin{equation*} 
\left\{ n-i\text{, }n | m_{i}+n-i-1=\min \left(
J\right) \text{ and/or }m_{n}+n-1=\min \left( J\right) \right\} \text{.}
\end{equation*}%
Which is $3$ in this case.

We apply $i_{\left( J\right) }=3$ squaring operations, namely:%
\begin{equation*}
Sq^{2^{\min \left( J\right) }}\text{, }Sq^{2^{\min \left( J\right) -1}}\text{%
, and }Sq^{2^{\min \left( J\right) }-2}\text{.}
\end{equation*}%
\begin{equation*}
Sq^{2^{\min \left( J\right)
}}d^{K}=d_{3}^{2+2^{2}}d_{1}^{2}d_{2}^{2^{3}+2^{4}}d_{1}^{2^{2}+2^{3}}+d_{3}^{2+2^{2}}d_{3}^{2^{3}}d_{2}^{2^{4}}d_{1}^{2^{2}+2^{3}}+d_{3}^{2+2^{2}}d_{2}^{2^{2}+2^{3}+2^{4}}d_{1}^{2^{3}}%
\text{.}
\end{equation*}%
\begin{gather*}
Sq^{2^{\min \left( J\right) -1}}\left[
d_{3}^{2+2^{2}}d_{2}^{2^{3}+2^{4}}d_{1}^{2+2^{2}+2^{3}}+d_{3}^{2+2^{2}+2^{3}}d_{2}^{2^{4}}d_{1}^{2^{2}+2^{3}}+d_{3}^{2+2^{2}}d_{2}^{2^{2}+2^{3}+2^{4}}d_{1}^{2^{3}}%
\right] \\
=d_{3}^{2+2^{2}}d_{2}^{2}d_{2}^{2^{3}+2^{4}}d_{1}^{2^{2}+2^{3}}+d_{3}^{2+2^{2}+2^{2}}d_{2}^{2^{3}+2^{4}}d_{1}^{2^{3}}
\end{gather*}%
\begin{equation*}
Sq^{2^{\min \left( J\right) -2}}\left[
d_{3}^{2+2^{2}}d_{2}^{2+2^{3}+2^{4}}d_{1}^{2^{2}+2^{3}}+d_{3}^{2+2^{2}+2^{2}}d_{2}^{2^{3}+2^{4}}d_{1}^{2^{3}}%
\right] =d_{3}^{2^{3}}d_{2}^{2^{3}+2^{4}}d_{1}^{2^{2}+2^{3}}\text{.}
\end{equation*}%
Finally,%
\begin{equation*}
Sq^{2^{3-2}}Sq^{2^{3-1}}Sq^{2^{3}}d^{K}=d_{3}^{2^{3}}d_{2}^{2^{3}+2^{4}}d_{1}^{2^{2}+2^{3}}%
\text{.}
\end{equation*}

Let $K=(2^{3}+2^{4},2^{2}+2^{3},2^{3})$. Then $K=(2^{3},2^{2},2^{3})$ , $%
\min \left( J\right) =3$, and \ $i_{\left( J\right) }=2$.%
\begin{equation*}
Sq^{2^{2}}Sq^{2^{3}}d^{K}=d_{3}^{2^{2}+2^{3}}d_{2}^{2^{3}+2^{4}}d_{1}^{2^{3}}%
\text{.}
\end{equation*}

Applying the procedure described above five more times we get:\newline
Let $Sq(\Gamma ,K)$\ be the following operation 
\begin{equation*}
Sq(2^{6},3)Sq(2^{4},1)Sq(2^{4},2)Sq(2^{4},3)Sq(2^{3},1)Sq(2^{3},2)Sq(2^{3},3)%
\text{,}
\end{equation*}%
then 
\begin{equation*}
Sq(\Gamma
,K)d_{3}^{2+2^{2}}d_{2}^{2^{3}+2^{4}}d_{1}^{2^{2}+2^{3}}=d_{3}^{2^{6}}\text{.%
}
\end{equation*}
\end{example}

\begin{definition}
\label{Def1}Let $K=(k\left( 1\right) ,...,k\left( n\right) )$ and $%
d^{K}=\prod\limits_{1}^{n}d_{n,i}^{k\left( i\right) }$. For each $k\left(
i\right) $, let $a_{i}p^{m\left( k\left( i\right) \right) }$ be its lowest
non-zero term in its $p$-adic expansion. Here $a_{i}=0$, if $k\left(
i\right) =0$. Let 
\begin{equation*}
J=(a_{1}p^{m\left( k\left( 1\right) \right) },...,a_{n}p^{m\left( k\left(
n\right) \right) })\text{,}
\end{equation*}%
and $\min \left( J\right) :$ 
\begin{equation*}
\min \{m\left( k\left( n\right) \right) +n-1,m\left( k\left( i\right)
\right) +n-i-1\;|1\leq i<n\text{ and \ }a_{n},a_{i}\neq 0\}\text{.}
\end{equation*}%
Let $i_{\left( J\right) }$ stand for the maximum of the $n-i$'s and/or $n$
such that $m\left( k\left( i\right) \right) +n-i-1=\min \left( J\right) $
and/or $m\left( k\left( n\right) \right) +n-1=\min \left( J\right) $.
\end{definition}

\begin{theorem}
\label{Theorem Steenr-Milnor}Let $d^{K}$, $\min \left( J\right) $\ and $%
i_{\left( J\right) }$\ as in the definition above, then 
\begin{equation*}
P\left( \min \left( J\right) ,i_{\left( J\right) }\right) (d^{K})=\left\{ 
\begin{array}{ll}
ud_{n,n}^{p^{m\left( k\left( n\right) \right) }}d^{K}\text{,} & \text{if }%
i_{\left( J\right) }=n \\ 
ud_{n,n}^{p^{m\left( k\left( n-i_{\left( J\right) }\right) \right)
}}d^{K}d_{n,n-i_{\left( J\right) }}^{-p^{m\left( k\left( n-i_{\left(
J\right) }\right) \right) }}\text{,} & \text{if }i_{\left( J\right) }<n%
\end{array}%
\right. \text{.}
\end{equation*}%
Here\ $u$\ is a unit.
\end{theorem}

\begin{proof}
Let $b_{n-i_{\left( J\right) }}=a_{n-i_{\left( J\right) }}$, and $b_{t}=0$,
otherwise. Let $B=\left( b_{1},...,b_{n}\right) $. The second statement of
lemma \ref{SteenActionLemma3} implies that 
\begin{equation*}
P\left( \min \left( J\right) ,i_{\left( J\right) }\right) (d^{K})=\left(
P\left( \min \left( J\right) ,i_{\left( J\right) }\right) (d^{B})\right)
d^{K-B}\text{.}
\end{equation*}%
The first claim of lemma \ref{SteenActionLemma3} provides the claim. \ 
\end{proof}

\begin{remark}
Let $d^{K^{\prime }}=P\left( \min \left( J\right) ,i_{\left( J\right)
}\right) (d^{K})$, where $K$, $\min \left( J\right) $, and $i_{\left(
J\right) }$\ as in the Theorem above. Let $K^{\prime }=(k^{\prime }\left(
1\right) ,...,k^{\prime }\left( n\right) )$, then $k\left( n\right)
<k^{\prime }\left( n\right) $ and $k\left( i\right) \geq k^{\prime }\left(
i\right) $ for $i<n$. \ 
\end{remark}

\begin{corollary}
Let $d^{K}=\prod\limits_{1}^{n}d_{n,i}^{k\left( i\right) }$ such that $%
\sum\limits_{1}^{n-1}k\left( i\right) >0$. Then there exists a sequence $%
P\left( K\right) $ of Steenrod operations such that 
\begin{equation*}
P\left( K\right) d^{K}=d^{L}
\end{equation*}%
Here $d^{L}=\prod\limits_{1}^{n}d_{n,i}^{l\left( i\right) }$\ satisfies $%
k\left( n\right) <l\left( n\right) $, $k\left( i\right) >l\left( i\right) $
for some $i<n$ and $k\left( t\right) =l\left( t\right) $ for $t\neq i,n$.
\end{corollary}

\begin{proof}
The hypothesis $\sum\limits_{1}^{n-1}k\left( i\right) >0$ implies that case
(2) of Theorem above will be applied at some stage of the procedure. So, for
some $i$, the corresponding exponent will be smaller in the new monomial.
\end{proof}

\begin{theorem}
\label{power of d-n,n}Let $d^{K}=\prod\limits_{1}^{n}d_{n,i}^{k\left(
i\right) }$. There exists a sequence of Steenrod operations $P\left( \Gamma
,K\right) $\ such that 
\begin{equation*}
P\left( \Gamma ,K\right) d^{K}=ud_{n,n}^{p^{m}}\text{.}
\end{equation*}%
for some natural number $m$ and unit $u$.
\end{theorem}

\begin{proof}
Last corollary is applied repeatedly, so the sequence of the exponents of
the resulting monomial will be $\left( 0,...,0,l\left( n\right) \right) $.
If $l\left( n\right) $ is not a $p$-th power, applying case (1) of last
Theorem repeatedly the exponent of $d_{n,n}$\ will become a $p$-th power.\ \ 
\end{proof}

Before proceeding to the proof that the classical Dickson algebra is atomic,
let us consider an example which illuminates the key ingredient for the
proof.

\begin{example}
Let $p=2\ $and $n=3$. Let $%
d^{K}=d_{1}^{2^{2}+2^{3}}d_{2}^{2^{3}+2^{4}}d_{3}^{2+2^{2}}$ and $%
d^{K^{\prime }}=d_{1}^{2^{2}+2^{4}}d_{3}^{2+2^{2}+2^{4}}$. Then $%
|d^{K}|=|d^{K^{\prime }}|$. As in the last example there exist sequences of
Steenrod operations $Sq\left( \Gamma ,K\right) $ and $Sq\left( \Gamma
,K^{\prime }\right) $\ such that 
\begin{equation*}
Sq(\Gamma ,K)d^{K}=d_{3}^{2^{6}}=Sq\left( \Gamma ^{\prime },K^{\prime
}\right) d^{K^{\prime }}.
\end{equation*}%
We recall that $Sq\left( \Gamma ,K\right) =$ 
\begin{equation*}
Sq(2^{6},3)Sq(2^{4},1)Sq(2^{4},2)Sq(2^{4},3)Sq(2^{3},1)Sq(2^{3},2)Sq(2^{3},3)
\end{equation*}%
and 
\begin{equation*}
Sq\left( \Gamma ^{\prime },K^{\prime }\right)
=Sq(2^{6},3)Sq(2^{5},2)Sq(2^{4},3)Sq(2^{3},2)Sq(2^{3},3).
\end{equation*}%
But 
\begin{equation*}
Sq(\Gamma ,K)d^{K^{\prime }}=0.
\end{equation*}
\end{example}

\begin{definition}
A graded module $\mathcal{M}$ is called atomic, if given any degree
preserving module map $f:\mathcal{M}\rightarrow \mathcal{M}$\ which is an
isomorphism on the lowest positive degree, then $f$ is an isomorphism.
\end{definition}

It turns out that the classical Dickson algebra is atomic as a Steenrod
algebra module.

\begin{corollary}
\label{Th2}Let $f:D_{n}\rightarrow D_{n}$ be an $\mathcal{A}$-linear map of
degree $0$ such that $f(d_{n,1})\neq 0$. Then $f$ is an isomorphism.
\end{corollary}

\begin{proof}
By hypothesis and proposition \ref{actionp-th}, $f(d_{n,i})=\lambda d_{n,i}$
for $i=1,...,n$ after applying a suitable Steenrod operation.

Let $d^{K}$ and $f(d^{K})=0$, then according to last Theorem there exists a
sequence of Steenrod operations such that $P\left( \Gamma ,K\right)
d^{K}=ud_{n,n}^{p^{m}}$. Thus $P\left( \Gamma ,K\right) f\left( d^{K}\right)
=uf\left( d_{n,n}^{p^{m}}\right) $ and $0=ud_{n,n}^{p^{m}}$.

Let homogeneous monomials $d^{K(t)}$ for $1\leq t\leq l$\ and $f(\sum
a_{t}d^{K\left( t\right) })=0$. Let $P\left( \Gamma \left( t\right) ,K\left(
t\right) \right) $ be the appropriate corresponding sequences of Steenrod
operations as in last Theorem: 
\begin{equation*}
P\left( \Gamma \left( t\right) ,K\left( t\right) \right)
=\prod\limits_{1}^{m_{t}}P\left( c_{s}\left( t\right) ,k_{s}\left( t\right)
\right) .
\end{equation*}%
Without lost of generality, we suppose that at least one of the 
\begin{equation*}
P\left( c_{s}\left( 1\right) ,k_{s}\left( 1\right) \right) \text{'s}
\end{equation*}%
is different. Otherwise, the common part is applied on the $d^{K}$'s and we
consider the new terms. Let $P\left( \Gamma \left( 1\right) ,K\left(
1\right) \right) =\prod\limits_{1}^{m_{t}}P\left( c_{s}\left( 1\right)
,k_{s}\left( 1\right) \right) $ satisfy the properties $c_{1}\left( 1\right)
=\min \left\{ c_{1}\left( t\right) \ |\ 1\leq t\leq l\right\} $ and if there
are more than one, then $k_{1}\left( 1\right) $\ is the biggest among the
equal ones.

If $c_{1}\left( 1\right) <c_{1}\left( t\right) $, then 
\begin{equation*}
P\left( \Gamma \left( 1\right) ,K\left( 1\right) \right) f(\sum
a_{t}d^{K\left( t\right) })=a_{1}P\left( \Gamma \left( 1\right) ,K\left(
1\right) \right) d^{K\left( 1\right) }
\end{equation*}%
and the first part of the proof provides a contradiction.

If $c_{1}\left( 1\right) =c_{1}\left( 2\right) $, then $k_{1}\left( 1\right)
>k_{1}\left( 2\right) $ according to our assumption. Thus (please consider
last example) 
\begin{equation*}
P\left( c_{1}\left( 1\right) ,k_{1}\left( 1\right) \right) d^{K\left(
2\right) }=0\neq P\left( c_{1}\left( 1\right) ,k_{1}\left( 1\right) \right)
d^{K\left( 1\right) }.
\end{equation*}%
This is because 
\begin{equation*}
P\left( c_{1}\left( 1\right) ,k_{1}\left( 1\right) \right) =P\left(
c_{1}\left( 1\right) -k_{1}\left( 2\right) +1,k_{1}\left( 1\right)
-k_{1}\left( 2\right) \right) P\left( c_{1}\left( 2\right) ,k_{1}\left(
2\right) \right)
\end{equation*}%
and $c_{1}\left( 1\right) -k_{1}\left( 2\right) +1<c_{2}\left( 2\right) $.
By lemma \ref{SteenActionLemma3}: 
\begin{equation*}
P\left( c_{1}\left( 1\right) -k_{1}\left( 2\right) +1,k_{1}\left( 1\right)
-k_{1}\left( 2\right) \right) P\left( c_{1}\left( 2\right) ,k_{1}\left(
2\right) \right) d^{K\left( 2\right) }=0.
\end{equation*}

And the first part of the proof provides a contradiction.\ 
\end{proof}

Next we show that the ring of upper triangular invariants does not have this
property.

\begin{example}
\label{antiparad}Let $p=2$ and $H_{2}=P[y_{1},y_{2}]^{U_{2}}$\ the ring of
upper triangular invariants. Here $U_{2}=\left\{ \left( 
\begin{array}{cc}
1 & a \\ 
0 & 1%
\end{array}%
\right) \ |\ a=%
%TCIMACRO{\TeXButton{TeX field}{\mathbb Z}}%
%BeginExpansion
\mathbb Z%
%EndExpansion
/2%
%TCIMACRO{\TeXButton{TeX field}{\mathbb Z}}%
%BeginExpansion
\mathbb Z%
%EndExpansion
\right\} $. It is known that it is a polynomial algebra on $h_{1}=y_{1}$ and 
$h_{2}=y_{2}^{2}+y_{2}y_{1}$ (\cite{Mui}).\ Let $f:H_{2}\rightarrow H_{2}$
be an $\mathcal{A}$-linear map such that $f(h_{1})=h_{1}$. Since $%
Sq^{1}h_{1}=h_{1}^{2}\neq h_{2}$, $f\left( h_{2}\right) $\ can be defined
independently of $h_{1}$: $f\left( h_{2}\right) =ah_{2}+bh_{1}^{2}$\ with $%
a,b\in 
%TCIMACRO{\TeXButton{TeX field}{\mathbb Z}}%
%BeginExpansion
\mathbb Z%
%EndExpansion
/2%
%TCIMACRO{\TeXButton{TeX field}{\mathbb Z}}%
%BeginExpansion
\mathbb Z%
%EndExpansion
$. Even if $f\left( h_{2}\right) =h_{1}^{2}$, $f$\ is not an isomorphism: $%
f\left( d_{2,1}\right) =f(h_{2}+h_{1}^{2})=0=f(d_{2,0})=f(Sq^{1}d_{2,1})$.
\end{example}

\section{Dickson algebras are atomic}

In this section the previous results are extended to the extended Dickson
algebra.

Mui gave an invariant theoretic description of the cohomology algebra of the
symmetric group and calculated rings of invariants involving the exterior
algebra $E(x_{1},\cdots ,x_{n})$ as well in \cite{Mui}. We recall that $%
|x_{i}|=1$ and $\beta x_{i}=y_{i}$.

\begin{theorem}
\cite{Mui}The extended Dickson algebra 
\begin{equation*}
D_{n}:=(E(x_{1},\cdots ,x_{n})\otimes P[y_{1},\cdots ,y_{n}])^{GL_{n}}
\end{equation*}%
is a tensor product of the polynomial algebra $P[y_{1},\cdots
,y_{n}]^{GL_{n}}$ and the 
%TCIMACRO{\TeXButton{TeX field}{$\mathbb Z$}}%
%BeginExpansion
$\mathbb Z$%
%EndExpansion
$/p$%
%TCIMACRO{\TeXButton{TeX field}{$\mathbb Z$}}%
%BeginExpansion
$\mathbb Z$%
%EndExpansion
-module spanned by the set of elements consisting of the following
monomials: 
\begin{equation*}
M_{n;s_{1},...,s_{m}}L_{n}^{p-2};\ 1\leq m\leq n,\ \text{and }0\leq
s_{1}<\dots <s_{m}\leq n-1.
\end{equation*}%
Its algebra structure is determined by the following relations:\newline
a) $(M_{n;s_{1},...,s_{m}}L_{n}^{p-2})^{2}=0$ for$\ 1\leq m\leq n,\ $and $%
0\leq s_{1}<\dots <s_{m}\leq n-1$.\newline
b) $%
M_{n;s_{1},...,s_{m}}L_{n}^{(p-2)}d_{n,1}^{m-1}=(-1)^{m(m-1)/2}M_{n;s_{1}}L_{n}^{p-2}\dots M_{n;s_{m}}L_{n}^{p-2} 
$.\newline
Here $1\leq m\leq n$, and $0\leq s_{1}<\dots <s_{m}\leq n-1$.
\end{theorem}

The elements $M_{n;s_{1},...,s_{m}}$ above have been defined by Mui in \cite%
{Mui} as follows: 
\begin{equation*}
M_{n;s_{1},...,s_{m}}=\frac{1}{m!}\left\vert 
\begin{array}{ccccc}
x_{1} &  & \cdots &  & x_{1} \\ 
\vdots &  &  &  & \vdots \\ 
x_{1} &  & \cdots &  & x_{n} \\ 
y_{1} &  & \cdots &  & y_{n} \\ 
\vdots &  &  &  & \vdots \\ 
y_{1}^{p^{n-1}} &  & \cdots &  & y_{n}^{p^{n-1}}%
\end{array}%
\right\vert \text{,}L_{n}=\left\vert 
\begin{array}{ccccc}
y_{1} &  & \cdots &  & x_{n} \\ 
y_{1}^{p} &  & \cdots &  & y_{n}^{p} \\ 
\vdots &  &  &  & \vdots \\ 
y_{1}^{p^{n-1}} &  & \cdots &  & y_{n}^{p^{n-1}}%
\end{array}%
\right\vert
\end{equation*}%
Here there are $m$ rows of $x_{i}$'s and the $s_{i}$-th's powers are
omitted, where $0\leq s_{1}<\dots <s_{m}\leq n-1$ in the first determinant.

The degrees of elements above are $%
|M_{n;s_{1},...,s_{m}}(L_{n})^{p-2}|=m+2((p^{n}-1)-(p^{s_{1}}+\dots
+p^{s_{m}}))$.

Next, some important subalgebras of $D_{n}$ are defined.

\begin{definition}
\label{Def SD-k}Let $SD_{n}$ be the subalgebra of $D_{n}$ generated by: 
\begin{equation*}
d_{n,s+1}\text{, }M_{n;s}(L_{n})^{p-2}\text{ and }%
M_{n;s_{1},s_{2}}(L_{n})^{p-2}\text{. }
\end{equation*}%
Here $0\leq s\leq n-1$. $0\leq s_{1}<s_{2}\leq n-1$.
\end{definition}

$D_{n}$ and $SD_{n}$\ are $\mathcal{A}$-algebras. It is known that $SD_{n}$\
is related with the $\hom $-dual of the length $n$\ coalgebra $R[n]$ of the
Dyer-Lashof algebra $R$\ (\cite{Kech3}).\ \ 

\begin{definition}
\label{Ideal In}Let $I_{n}$ stand for the ideal of $SD_{n}$ generated by 
\begin{equation*}
\{d_{n,n}\text{, }M_{n;i}(L_{n})^{p-2}\text{ and }M_{n;0,i}(L_{n})^{p-2}\ |\
0\leq i\leq n-1\}.
\end{equation*}
\end{definition}

The ideal $I_{n}$\ is related with the $\hom $-dual of the length $n$ module
of indecomposable elements of $H_{\ast }(Q_{0}S^{0};%
%TCIMACRO{\TeXButton{Z-p}{\Bbb{Z}/\Bbb{Z}p}}%
%BeginExpansion
\Bbb{Z}/\Bbb{Z}p%
%EndExpansion
)$ (\cite{C-P-S1}).\ 

First, we recall the next proposition concerning the action of the Steenrod
algebra generators on exterior generators on the extended Dickson algebra.
For the benefit of the reader we also recall that $P^{p^{k}}L_{n}^{p-2}=0$
for $0\leq k\leq n-2$.

\begin{proposition}
\cite{Kech3}%
\begin{equation*}
\beta M_{n;0}L_{n}^{n-2}=d_{n,n};
\end{equation*}%
\begin{equation*}
\beta M_{n;0,s}L_{n}^{n-2}=-M_{n;s}L_{n}^{n-2}\text{, for }n-1\geq s>0;
\end{equation*}%
\begin{equation*}
P^{p^{s-1}}M_{n;s}L_{n}^{n-2}=M_{n,s-1}L_{n}^{n-2}\text{, for }n-1\geq s>0.
\end{equation*}
\end{proposition}

\begin{remark}
Propositions \ref{actionp-th} and the last one imply that $I_{n}$ is closed
under the Steenrod algebra action.
\end{remark}

In the next lemmata we explain step by step how a monomial is transformed in
to a power of $d_{n,n}$. We start with an application of the proposition
above.

\begin{lemma}
\label{Bhta apo 0}%
\begin{equation*}
\beta P\left( i-1,i\right)
M_{n;i,s_{1},...,s_{l}}L_{n}^{n-2}=-M_{n;s_{1},...,s_{l}}L_{n}^{n-2}\text{.}
\end{equation*}%
Here $0\leq i<s_{1}<...<s_{l}\leq n-1$.%
\begin{equation*}
\beta P\left( t-1,t\right) M_{n;i}L_{n}^{n-2}=\left\{ 
\begin{array}{cc}
d_{n,n} & \text{for }t=i \\ 
0 & \text{for }t\neq i%
\end{array}%
\right. \text{.}
\end{equation*}%
\begin{equation*}
\beta \underbrace{P^{p^{0}}\beta }...\underbrace{P^{p^{l-2}}...P^{p^{0}}%
\beta }M_{n;0,1,...,l-1}L_{n}^{p-2}=ud_{n,n}\text{.}
\end{equation*}
\end{lemma}

\begin{definition}
Let $M=M_{n;s_{1},...,s_{l}}L_{n}^{p-2}$ and $\left(
t_{1},...,t_{n-l}\right) $\ be the support of $\left( s_{1},...,s_{l}\right) 
$\ in $\left\{ 0,1,...,n-1\right\} $. Let $P\left( B\left( M\right) \right)
:=$%
\begin{equation*}
\beta \underbrace{P^{p^{0}}\beta }...\underbrace{P^{p^{l-2}}...P^{p^{0}}%
\beta }\underbrace{P^{p^{l-1}}...P^{p^{t_{1}}}}...\underbrace{%
P^{p^{l+k-2}}...P^{p^{t_{k}}}}...\underbrace{P^{p^{n-2}}...P^{p^{t_{n-l}}}}%
\text{.}
\end{equation*}
\end{definition}

\begin{lemma}
\label{P(B(M))}Let $M=M_{n;s_{1},...,s_{l}}L_{n}^{p-2}$, then 
\begin{equation*}
P\left( B\left( M\right) \right) M=ud_{n,n}\text{.}
\end{equation*}
\end{lemma}

\begin{proof}
We recall that $M_{n;s_{1},...,s_{l}}$ is a sum of monomials of the form $%
x_{1}...x_{l}y_{l+1}^{p^{t_{1}}}...y_{n}^{p^{t_{n-l}}}$.\ Let $s_{l}=n-1$.
Then 
\begin{equation*}
P^{p^{n-2}}...P^{p^{t_{n-l}}}y_{l+1}^{p^{t_{1}}}...y_{n}^{p^{t_{n-l}}}=y_{l+1}^{p^{t_{1}}}...y_{n-1}^{p^{t_{n-l-1}}}y_{n}^{p-2}.
\end{equation*}%
Hence $%
P^{p^{n-2}}...P^{p^{t_{n-l}}}M=M_{n;s_{1},...,s_{l-1},t_{n-l}}L_{n}^{p-2}$.

Let $s_{l}<n-1$ and $k$\ maximal such that $t_{k}<t_{k+1}-1$ and $k<n-l$. In
this case%
\begin{equation*}
P^{p^{t_{k+1}-2}}...P^{p^{t_{k}}}y_{l+1}^{p^{t_{1}}}...y_{n}^{p^{t_{n-l}}}=y_{l+1}^{p^{t_{1}}}...y_{n}^{p^{t_{n-l}}}y_{k}^{p^{t_{k+1}-1}}y_{k}^{-p^{t_{k}}}.
\end{equation*}%
Hence $P^{p^{t_{k+1}-2}}...P^{p^{t_{k}}}M=M_{n;s_{1},...,t_{k},\widehat{%
t_{k+1}-1},...,s_{l}}L_{n}^{p-2}$. Here $\widehat{t_{k+1}-1}$ means that the
index $t_{k+1}-1$ is missing. Now the claim follows.
\end{proof}

The main point of this section is to prove that, if $M\neq M^{\prime }$,
then $P\left( B\left( M\right) \right) M^{\prime }=0$. The next example
demonstrates the idea.

\begin{example}
1) $P^{p^{7}}P^{p^{6}}M_{10;4,7,8}L_{10}^{p-2}=M_{10;4,6,7}L_{10}^{p-2}$.

$%
P^{p^{7}}P^{p^{6}}M_{10;3,5,7}L_{10}^{p-2}=P^{p^{7}}M_{10;3,5,6}L_{10}^{p-2}=0 
$.

2) $P^{p^{8}}M_{10;3,5,9}L_{10}^{p-2}=M_{10;3,5,8}L_{10}^{p-2}$.

$P^{p^{8}}M_{10;3,5,7}L_{10}^{p-2}=0$.
\end{example}

\begin{proposition}
Each element $M=M_{n;s_{1},...,s_{l}}L_{n}^{p-2}$ uniquely determines $%
P\left( B\left( M\right) \right) $ such that $P\left( B\left( M\right)
\right) M=ud_{n,n}$ and $P\left( B\left( M\right) \right) M^{\prime }=0$\
for $M^{\prime }=M_{n;s_{1}^{\prime },...,s_{l^{\prime }}^{\prime
}}L_{n}^{p-2}\neq M$. \ 
\end{proposition}

\begin{proof}
Let $\left( t_{1},...,t_{n-l}\right) $ and $\left( t_{1}^{\prime
},...,t_{n-l}^{\prime }\right) $ \ be the corresponding supports. We recall
that $M_{n;s_{1},...,s_{l}}$ is a sum of monomials of the form $%
x_{1}...x_{l}y_{l+1}^{p^{t_{1}}}...y_{n}^{p^{t_{n-l}}}$.

Let $s_{l}=n-1>s_{l}^{\prime }$. Then $%
P^{p^{n-2}}...P^{p^{t_{n-l}}}y_{l+1}^{p^{t_{1}^{\prime
}}}...y_{n}^{p^{t_{n-l}^{\prime }}}$\ is either zero (if $t_{n-l}\neq
t_{q}^{\prime }$) or contains two identical powers $p^{m-1}$\ and in either
case the determinant $P^{p^{n-2}}...P^{p^{t_{n-l}}}M^{\prime }$\ is zero.

Let $s_{l}^{\prime }<s_{l}<n-1$ and $k$\ maximal such that $t_{k}<t_{k+1}-1$
and $k<n-l$.

If $t_{k}\notin \left( t_{1}^{\prime },...,t_{n-l}^{\prime }\right) $, then $%
P^{p^{t_{k}}}M^{\prime }=0$ (please recall the proof of the last lemma). Let 
$t_{k}\in \left( t_{1}^{\prime },...,t_{n-l}^{\prime }\right) $ and $%
t_{k}^{\prime }=t_{k}$.\ If $t_{k+1}^{\prime }=t_{k}^{\prime }+1$, then $%
P^{p^{t_{k}}}y_{l+1}^{p^{t_{1}^{\prime }}}...y_{n}^{p^{t_{n-l}^{\prime }}}$
contains two identical powers $p^{t_{k}^{\prime }+1}$\ and the determinant $%
P^{p^{t_{k}}}M^{\prime }$\ is zero.

If $t_{k+1}^{\prime }>t_{k}^{\prime }+1$, then $t_{k+1}-t_{k}>t_{k+1}^{%
\prime }-t_{k}^{\prime }$. Again for the same reason $%
P^{p^{t_{k+1}-2}}...P^{p^{t_{k}}}M^{\prime }=0$.\ 

Let $s_{l}<s_{l}^{\prime }$, then 
\begin{equation*}
P\left( B\left( M\right) \right) =(\beta P^{p^{0}}\beta
...P^{p^{l-2}}...P^{p^{0}}\beta )P^{p^{i_{q}}}...P^{p^{i_{1}}}
\end{equation*}%
\ and 
\begin{equation*}
P\left( B\left( M^{\prime }\right) \right) =(\beta P^{p^{0}}\beta
...P^{p^{l-2}}...P^{p^{0}}\beta )P^{p^{i_{q^{\prime }}^{\prime
}}}...P^{p^{i_{1}^{\prime }}}
\end{equation*}%
such that $i_{1}<i_{1}^{\prime }$ (please see lemma \ref{P(B(M))}). Now 
\begin{equation*}
P^{p^{i_{q}}}...P^{p^{i_{1}}}M=M_{n;0,1,...,l-1}L_{n}^{p-2}=P^{p^{i_{q^{%
\prime }}^{\prime }}}...P^{p^{i_{1}^{\prime }}}M^{\prime }
\end{equation*}%
If $P^{p^{i_{q}}}...P^{p^{i_{1}}}M^{\prime }\neq 0$, then $%
P^{p^{i_{q}}}...P^{p^{i_{1}}}M^{\prime }\neq M_{n;0,1,...,l-1}L_{n}^{p-2}$,
because $i_{1}<i_{1}^{\prime }$. Thus $P\left( B\left( M\right) \right)
M^{\prime }=0$.

If $s_{l}=s_{l}^{\prime }$, then by applying a suitable sequence of Steenrod
operations the case is reduced to one of the previous ones.\ 
\end{proof}

\begin{corollary}
\label{Steenrod action exterior}Let $M=M_{n;s_{1},...,s_{l}}L_{n}^{p-2}d^{K}$
be a monomial in $D_{n}$, then $P\left( B\left( M\right) \right)
M=ud_{n,n}d^{K}$. Let $M^{\prime }=M_{n;s_{1}^{\prime },...,s_{l^{\prime
}}^{\prime }}L_{n}^{p-2}d^{K^{\prime }}$ such that $s_{t}\neq s_{t}^{\prime
} $ for some $t$, then $P\left( B\left( M\right) \right) M^{\prime }=0$.
\end{corollary}

\begin{proof}
For the first statement, last lemma implies that 
\begin{equation*}
P\left( B\left( M\right) \right) M=ud_{n,n}d^{K}+\left( others\right) .
\end{equation*}%
Last proposition implies that $\left( others\right) =0$. The second
statement is an application of last proposition. \ \ 
\end{proof}

Now we are ready to proceed to our main results of this section.

\begin{theorem}
\label{Th4}The extended Dickson algebra $D_{n}$ is atomic.
\end{theorem}

\begin{proof}
Let $g:D_{n}\rightarrow D_{n}$ be an $\mathcal{A}$-linear map such that $%
g(M)\neq 0$. We will prove that $g$ is an isomorphism. Here $%
M=\prod\limits_{1}^{n}x_{i}L_{n}^{p-2}$. Applying corollaries \ref{Steenrod
action exterior}\ and \ref{Th2},\ the claim is obtained.
\end{proof}

\begin{proposition}
\label{SD_k atomic}a) $SD_{n}$\ is atomic as a Steenrod algebra module.\ \ \ 

b) $I_{n}$ is atomic as a Steenrod algebra module.
\end{proposition}

\begin{proof}
a) Let $f:SD_{n}\rightarrow SD_{n}$ satisfy 
\begin{equation*}
f(M_{n;n-2,n-1}L_{n}^{p-2})=uM_{n;n-2,n-1}L_{n}^{p-2}\text{.}
\end{equation*}%
Applying corollaries \ref{Steenrod action exterior}\ and \ref{Th2},\ the
claim is obtained.

b) Let us recall that $I_{n}$ is the ideal generated by 
\begin{equation*}
\{d_{n,n},M_{n;s_{1}}L_{n}^{p-2},M_{n;0,s_{1}^{\prime }}L_{n}^{p-2}\}.
\end{equation*}%
Here $0\leq s_{1}<n$ and $0<s_{1}^{\prime }\leq n-1$. Let $f$ satisfy 
\begin{equation*}
f(M_{n;0,n-1}L_{n}^{p-2})=uM_{n;0,n-1}L_{n}^{p-2}.
\end{equation*}%
The proof follows the same pattern as above.
\end{proof}

Proposition b) above is a reformulation of Theorem 4.1 in \cite{C-P-S1}.

We close this section by observing a property of the Steenrod algebra. We
recall that an $\mathcal{A}$-module is indecomposable, if it is not a
non-trivial direct sum.

Let $\overline{D_{n}}$\ denote the augmentation ideal of $D_{n}$.

\begin{corollary}
\label{decomposit}$\overline{D_{n}}$ is not directly decomposable as an \ $%
\mathcal{A}$-module.\ 
\end{corollary}

\begin{proof}
Assume $\overline{D_{n}}=\bigoplus\limits_{i\in I}(D_{n})_{i}$ such that $%
(D_{n})_{i}\neq 0$. If $d\left( i\right) $\ and $d\left( j\right) $\ are
homogeneous polynomials in $(D_{n})_{i}$\ and $(D_{n})_{j}$\ respectively,
then there exist $P^{\Gamma }$\ and $P^{\Gamma ^{\prime }}$\ such that $%
a_{i}P^{\Gamma }d\left( i\right) =d_{n,n}^{p^{l}}=b_{j}P^{\Gamma }d\left(
j\right) $. \ 
\end{proof}

\section{$Q_{0}S^{0}$ is $H$-atomic at $p$}

We close this work by applying the main result in the $mod-p$ homology
of $QS^{0}$. $H_{\ast }(Q_{0}S^{0};%
%TCIMACRO{\TeXButton{Z-p}{\Bbb{Z}/\Bbb{Z}p}}%
%BeginExpansion
\Bbb{Z}/\Bbb{Z}p%
%EndExpansion
)$ is described in terms of Dyer-Lashof operations. For their properties
please see May \cite{CLM}.\ 

Iterates of the Dyer-Lashof operations are of the form $Q^{\left(
I,\varepsilon \right) }=\beta ^{\epsilon _{1}}Q^{i_{1}}\dots \beta
^{\epsilon _{k}}Q^{i_{k}}$ where $(I,\varepsilon )=((i_{k},\dots
,i_{1}),(\epsilon _{k},...,\epsilon _{1}))$ with $\epsilon _{j}=0\ $or$\ 1$
and $i_{j}$ a non-negative integer for $j=1,\ \dots \ ,k$. If $p=2$, $%
\epsilon _{j}=0\ $for all $j$. The degree is defined by $|\left(
I,\varepsilon \right) |:=|Q^{\left( I,\varepsilon \right) }|=2(p-1)\left(
\sum\limits_{t=1}^{k}i_{t}\right) -\left(
\sum\limits_{t=1}^{k}e_{t}\right) $ [$|Q^{I,\varepsilon }|=\left(
\sum\limits_{t=1}^{k}i_{t}\right) $, for $p=2$]. Let $l(I,\varepsilon )=k$
denote the length of $\left( I,\varepsilon \right) $ or $Q^{I,\varepsilon }$
and let the excess of $(I,\varepsilon )$ or $Q^{\left( I,\varepsilon \right)
}$, denoted by $exc(Q^{\left( I,\varepsilon \right) })=i_{k}-\epsilon
_{k}-|Q^{(I\left( k-1\right) ,\varepsilon \left( k-1\right) )}|$, where $%
(I\left( t\right) ,\varepsilon \left( t\right) )=((i_{t},\dots
,i_{1}),(\epsilon _{t},\dots ,\epsilon _{1}))$. 
\begin{equation*}
exc(Q^{I,\varepsilon })=i_{k}-\epsilon _{k}-2(p-1)\sum\limits_{1}^{k-1}i_{t}%
\text{, [}exc(Q^{I})=i_{k}-\sum\limits_{1}^{k-1}i_{t}\text{].}
\end{equation*}%
The excess is defined $\infty $, if $I=\emptyset $ and we omit the sequence $%
(\epsilon _{1},...,\epsilon _{k})$, if all $\epsilon _{i}=0$. We refer to
elements $Q^{I}$ as having non-negative excess, if $exc(Q^{(I\left( t\right)
,\varepsilon \left( t\right) )})$ is non-negative for all $t$.

There are relations among the iterated operations called Adem relations, so
an operation can be reduced to a sum of \textit{admissible} operations after
applying Adem relations.\ A sequence $(I,\varepsilon )$ is called
admissible, if $pi_{j}-\epsilon _{j}\geq i_{j-1}$ ($2i_{j}\geq i_{j-1}$) for 
$2\leq j\leq k$.

The Kronecker pairing and the left $\mathcal{A}$-module on $H^{\ast
}(Q_{0}S^{0};%
%TCIMACRO{\TeXButton{Z-p}{\Bbb{Z}/\Bbb{Z}p}}%
%BeginExpansion
\Bbb{Z}/\Bbb{Z}p%
%EndExpansion
)$ induces a right $\mathcal{A}_{\ast }$-module\ structure on $H_{\ast
}(Q_{0}S^{0};%
%TCIMACRO{\TeXButton{Z-p}{\Bbb{Z}/\Bbb{Z}p}}%
%BeginExpansion
\Bbb{Z}/\Bbb{Z}p%
%EndExpansion
)$. We follow Cohen and May in writing the Steenrod operations on the left.
\ 
 
The Dyer-Lashof algebra can be decomposed as opposite Steenrod coalgebras
with respect to length 
\begin{equation*}
R=\bigoplus\limits_{k\geq 0}R[k]\text{.}
\end{equation*}

Let $R^{+}$ be the positive degree elements of $R$ and $R_{0}$ be the ideal
generated by positive degree elements of excess zero 
\begin{equation*}
R_{0}=<Q^{\left( I,\varepsilon \right) }\ |\ exc(I,\varepsilon )=0>\text{.}
\end{equation*}

Let $(I,\varepsilon )$ be an admissible sequence such that $|Q^{\left(
I,\varepsilon \right) }|>0$, then $Q^{\left( I,\varepsilon \right) }[1]$
corresponds to 
\begin{equation*}
\lbrack 1]_{\left( I,\varepsilon \right) }:=Q^{\left( I,\varepsilon \right)
}[1]\ast <Q^{\left( I,\varepsilon \right) }[1]>^{-1}\in H_{\ast }(Q_{0}S^{0})%
\text{.}
\end{equation*}%
Here $<Q^{\left( I,\varepsilon \right) }[1]>^{-1}=[-p^{l(I)}]$.

According to Madsen and May, $H_{\ast }(Q_{0}S^{0};%
%TCIMACRO{\TeXButton{Z-p}{\Bbb{Z}/\Bbb{Z}p}}%
%BeginExpansion
\Bbb{Z}/\Bbb{Z}p%
%EndExpansion
)$ is the free commutative graded algebra generated by $Q^{(I,\varepsilon
)}[1]\ast \lbrack -p^{l(I)}]$. Here $(I,\varepsilon )$ are admissible
sequences of positive excess and $\ast $ denotes Pontryagin multiplication.
There exists an $\mathcal{A}_{\ast }$-module isomorphism between the
generators of $H_{\ast }(Q_{0}S^{0};%
%TCIMACRO{\TeXButton{Z-p}{\Bbb{Z}/\Bbb{Z}p}}%
%BeginExpansion
\Bbb{Z}/\Bbb{Z}p%
%EndExpansion
)$ and the quotient $R/Q_{0}R$ where $Q_{0}R=\{Q^{(I,\varepsilon
)}|exc(I,\varepsilon )=0\}$. It is known that $R[k]^{\ast }\cong SD_{k}$ as
Steenrod algebras (\cite{Mui 2}, \cite{Kech3}) and $(R/Q_{0}R)[k]^{\ast
}\cong I_{k}$ as Steenrod modules (\cite{C-P-S1}).

Next Theorem has been given in \cite{C-C-P-S1} as Theorem 2.3 for $p=2$ by a
similar method.

\begin{theorem}
\label{isomorphism in loops}\textit{Let }$f:Q_{0}S^{0}\rightarrow Q_{0}S^{0}$
be an $H$-map which induces an isomorphism on $H_{2p-3}(Q_{0}S^{0};%
%TCIMACRO{\TeXButton{Z-p}{\Bbb{Z}/\Bbb{Z}p}}%
%BeginExpansion
\Bbb{Z}/\Bbb{Z}p%
%EndExpansion
)$. Let $p>2$ and 
\begin{equation*}
f_{\ast }(Q^{\left( p,1\right) }[1])=uQ^{\left( p,1\right) }[1]+others
\end{equation*}%
for some $u\in (%
%TCIMACRO{\TeXButton{Z-p}{\Bbb{Z}/\Bbb{Z}p}}%
%BeginExpansion
\Bbb{Z}/\Bbb{Z}p%
%EndExpansion
)^{\ast }$\textit{. } Then $f_{\ast }$ is an isomorphism.
\end{theorem}

\begin{proof}
The case $p>2$\ shall be considered. We shall prove that $f_{\ast }$\ is an
isomorphism on\ $Q\left( H_{\ast }(Q_{0}S^{0};%
%TCIMACRO{\TeXButton{Z-p}{\Bbb{Z}/\Bbb{Z}p}}%
%BeginExpansion
\Bbb{Z}/\Bbb{Z}p%
%EndExpansion
)\right) $, the module of indecomposable elements. There is an $\mathcal{A}%
_{\ast }$\ module isomorphism between the previous module and $R/Q_{0}R$.
The last isomorphism provides an $\mathcal{A}$-isomorphism between $\left(
R/Q_{0}R\right) ^{\ast }$\ and $I=\oplus I_{k}$. Let $\overline{f}_{k}$\ be
the induced map of $f^{\ast }$ in $I_{k}$. It suffices to prove that $%
\overline{f}_{k}$ is an isomorphism for each $k$ and this is true as long as 
$\overline{f}_{k}(d)=ud$\ for $d=M_{k;0,s}L_{k}^{p-2}$\ or $d_{k,k}$ and $%
0<s\leq k-1$ according to proposition \ref{SD_k atomic}. Here $u$\ is a unit.

Given $f_{\ast }\left( \beta Q^{1}[1]\right) =u\beta Q^{1}[1]$\ we have 
\begin{equation*}
\beta f_{\ast }\left( Q^{1}[1]\right) =f_{\ast }\left( \beta Q^{1}[1]\right)
=u\beta Q^{1}[1]\text{.}
\end{equation*}%
Thus $f_{\ast }\left( Q^{1}[1]\right) =uQ^{1}[1]$, for degree reasons.
Moreover, 
\begin{equation*}
f_{\ast }\left( Q^{1}[1]\right) ^{p^{m}}=u\left( Q^{1}[1]\right)
^{p^{m}}\approx uQ^{\left( p^{m-1}-p^{m-2},...,p-1,1\right) }[1].
\end{equation*}

Dually (\cite{Kech2}), $\overline{f}_{1}\left( d_{1,1}\right) =ud_{1,1}$.
Given $f_{\ast }\left( Q^{(p,1)}[1]\right) =u^{\prime }Q^{(p,1)}[1]+others$,
we have $\overline{f}_{2}\left( d_{2,2}\right) =u^{\prime }d_{2,2}+others$.
Induction on the length $k$ is applied. Suppose that 
\begin{equation*}
f_{\ast }\left( Q^{\left( p^{k-1},...,p,1\right) }[1]\right) =uQ^{\left(
p^{k-1},...,p,1\right) }[1]+others.
\end{equation*}%
Now, $f_{\ast }\left( Q^{\left( p^{k-1},...,p,1\right) }[1]\right)
^{p}=u\left( Q^{\left( p^{k-1},...,p,1\right) }[1]\right) ^{p}+others$. And 
\begin{equation*}
f_{\ast }\left( Q^{\left( p^{k}-1,p^{k-1}...,p,1\right) }[1]\right)
=uQ^{\left( p^{k}-1,p^{k-1}...,p,1\right) }[1]+others.
\end{equation*}%
Since $k>2$, $P_{\ast }^{1}\left( Q^{\left( p^{k}-1,p^{k-1}...,p,1\right)
}[1]\right) =Q^{\left( p^{k},p^{k-1}...,p,1\right) }[1]$ uniquely by Nishida
relations. Thus 
\begin{equation*}
f_{\ast }\left( Q^{\left( p^{k},p^{k-1}...,p,1\right) }[1]\right)
=uQ^{\left( p^{k},p^{k-1}...,p,1\right) }[1]+others.
\end{equation*}%
Dually (\cite{Kech2}), $\overline{f}_{k+1}\left( d_{k+1,k+1}\right)
=u^{\prime }d_{k+1,k+1}+others$. Proposition \ref{SD_k atomic} implies that $%
\overline{f}_{k+1}$ is an isomorphism for all $k$. Now the Theorem follows.
\ 
\end{proof}

\subsection{Acknowledgement}

We express our profound thanks to H.-W. Henn, P. May and L. Schwartz.

\end{document}